\documentclass[12pt]{amsart}

\allowdisplaybreaks[1]
\usepackage{enumerate}
\usepackage{allrunes}
\usepackage{amsmath}
\usepackage{amssymb}
\usepackage{amsfonts}
\usepackage{verbatim}
\usepackage[dvipsnames]{xcolor}
\usepackage{graphicx}
\usepackage{float}
\usepackage{hyperref}
\makeindex

 \newtheorem{theorem}{Theorem}[section]
 \newtheorem{corollary}[theorem]{Corollary}
 \newtheorem{lemma}[theorem]{Lemma}

\newtheorem{example}[theorem]{Example}
\theoremstyle{definition}

\theoremstyle{remark}

\newtheorem{fact*}{Fact}

\DeclareMathOperator{\IM}{Im}

\DeclareMathOperator{\dist}{dist}

\DeclareMathOperator\im{\mathrm {Im~}}

\newcommand\dd{\mathrm d}
\newcommand{\eps}{\varepsilon}

\newcommand{\hilbert}{\mathcal{H}}

\renewcommand{\T}{\mathbb{T}}

\renewcommand{\D}{\mathbb{D}}
\newcommand{\C}{\mathbb{C}}

\renewcommand{\R}{\mathbb{R}}

\newcommand{\cc}[1]{\overline{#1}}
\newcommand{\abs}[1]{\left\vert#1\right\vert}

\newcommand{\nt}{\stackrel{\mathrm {nt}}{\to}}

\newcommand{\ip}[2]{\left\langle #1, #2 \right\rangle}

\newcommand{\inv}{^{-1}}

\newcommand{\til}{\raise.17ex\hbox{$\scriptstyle\mathtt{\sim}$}}

\newcommand{\ph}{\varphi}

\newcommand\ep{\varepsilon}
\newcommand\la{\lambda}

\newcommand\beq{\begin{equation}}

\newcommand\eeq{\end{equation}}

\newcommand\black{\color{black}}

\newcommand{\bbm}{\left[ \begin{smallmatrix}}
\newcommand{\ebm}{\end{smallmatrix} \right]}
\newcommand{\bpm}{\left( \begin{smallmatrix}}
\newcommand{\epm}{\end{smallmatrix} \right)}
\numberwithin{equation}{section}

\newlength{\Mheight}
\newlength{\cwidth}

\newcommand{\dfn}[1]{{\bf #1}\index{#1}}

\renewcommand{\k}{\mathfrak{k}}
\newcommand\fkf{\mathfrak{F}_f^\mathfrak{k}}
\newcommand\afkf{A\mathfrak{F}_{f, \la}^{\mathfrak{k}, \tau}}
\newcommand{\mk}{\mathfrak{k}}
\newcommand{\gandalf}{\textarc{F}}
\newcommand{\enigma}{\textarc{m}}
\newcommand{\ethereal}{\rotatebox[origin=c]{180}{\enigma}}

\title[Escaping Nontangentiality II]{Averaged mixed Julia-Fatou type theory with applications to spectral foliation}
\author{
J. E. Pascoe$^\dagger$
}
\address{Department of Mathematics\\
1400 Stadium Rd\\
  University of Florida\\
 Gainesville, FL 32611}
\email[J. E. Pascoe]{pascoej@ufl.edu}

\thanks{$\dagger$ Partially supported by National Science Foundation Mathematical
Science Postdoctoral Research Fellowship  
DMS 1606260 and DMS Analysis Grant 1953963}

\author{
Ryan Tully-Doyle$^\ddagger$
}
\address{Mathematics Department \\
California Polytechnic State University\\
San Luis Obispo, CA 93407 }
\email[R. Tully-Doyle]{rtullydo@calpoly.edu}
\thanks{$\ddagger$ Partially supported by National Science Foundation DMS Analysis Grant 2055098 and University of New Haven SRG and Cal Poly startup}

\date{\today}


\begin{document}

\begin{abstract}
Classically, theorems of Fatou and Julia describe the boundary regularity of functions in one complex variable. The former says that a complex analytic function on the disk has non-tangential boundary values almost everywhere, and the latter describes when a function takes an extreme value at a boundary point and is differentiable there non-tangentially. We describe a class of intermediate theorems in terms of averaged Julia-Fatou quotients. Boundary regularity is related to integrability of certain quantities against a special measure, the so-called Nevanlinna measure. Applications are given to spectral theory. 
\end{abstract}

\maketitle

\section{Introduction}

{ The goal of this manuscript is to develop a theory relating the boundary behavior of a class of analytic functions $f$ on regions of either nontangential or controlled tangential approach to the boundary on the one hand, and on the other the regularity of $f$ and a corresponding measure $\mu$, with a special focus on those properties which factor through conformal maps that are continuous at the boundary.}

\subsection{The classical setting}

The classical approach to boundary questions involves sets that approach the boundary {\bf nontangentially}. We define a \dfn{Stolz region} in a domain $D$ at a point $\tau \in \partial D$ with aperture $M,$ denoted $S_{\tau,M}$, to be the set
		$$S_{\tau ,M} = \{z\in D|\dist (z, \partial D) \geq M \dist(z, \tau)\}.$$
	Note that for the Stolz region $S_{\tau,M}$ to be nonempty, we must have $M\leq 1.$
A set $S$ is \dfn{nontangential at $\tau$} if it can be contained in a Stolz region. In the classical regimes on the disk, a Stolz region is a sector of a circle with point at $\tau$.

Fatou's Theorem \cite{fatou1906} concerns the existence of nontangential limits at the boundary for functions analytic on the unit disk $\D$:
\begin{theorem}
 Let $f:\D \to \C$ be a bounded analytic function. Then $f$ has a nontangential limit at a.e. $\tau \in \T$.
\end{theorem}

The Julia-Carath\'eodory theorem concerns the existence of nontangential boundary derivatives for analytic self-maps of the unit disk. The condition is given in terms of the classical Julia quotient
\[
		J_f(z) = \frac{\dist(f(z), \partial \Omega)}{\dist(z,\partial D)}. 
\]

\begin{theorem}[Julia-Carath\'eodory \cite{ju20, car29}]\label{jc}
Let $f: \D \to \cc\D$ be an analytic function. Let $\tau$ be a point in $\T = \partial \D$. The following are equivalent:
\begin{enumerate}
    \item There exists a sequence $\la_n \subset \D$ tending to $\tau$ such that $J_f(\la_n)$ is bounded as $\la_n \to \tau$;
    \item for every sequence $\la_n$ in $\D$ tending to $\tau$ nontangentially, the sequence $J_f(\la_n)$ is bounded;
    \item the function $f$ has a conformal linear approximation near $\tau$. That is, there exist an $\omega = \ph(\tau) \in \T$ and an $\eta = f'(\tau) \in \C$ so that 
    \[
    f(z) = \omega + \eta(z - \tau) + o(\abs{z - \tau}).
    \]
    as $\la \nt \tau$.
\end{enumerate}
\end{theorem}

Long lines of research continue the investigation and generalization of these results. Fatou type theorems have been investigated classically by Littlewood \cite{littlewood1927} and Zygmund \cite{zygmund49}, and more recently in work by Nagel, Rudin, and Shapiro \cite{nrs1982} and Nagel and Stein \cite{ns84}. Results investigating Julia questions can be found, e.g., in a body of work including \cite{nev22, ab98, jaf93, lugarnedic}. Other work on boundary regularity can be found in a series of papers by Bolotnikov and Kheifets (see, e.g.,  \cite{bk06}).

Points $\tau$ for which condition (2) in Theorem \ref{jc} is satisfied are called \emph{$B$-points}. They have been studied extensively in the two and several variable settings. See, e.g., \cite{amy10a, aty12, bps1, pascoePEMS, mprevisit, tdopmat}.

Throughout the current investigation, we will assume the axiom of choice.

\section{Synopsis of results}

\subsection{The averaged Julia-Fatou regime}
 In the prequel \cite{ajq} and in this work, we seek to establish a unified Julia-Fatou theory that allows controlled {\bf tangential} approach to the boundary by way of averaging boundary behavior in the Julia-Fatou quotients. The results in this paper begin with the \dfn{Pick class} of analytic self-maps of the complex upper half plane. Understanding boundary behavior of functions in the Pick class gives insight into measure theoretic questions by way of the classical Cauchy transform of positive measures on $\R$.

Let $\Pi$ denote the upper half plane in $\mathbb{C}.$ Let $f: \Pi \rightarrow \cc\Pi$ be an analytic function. Let $\mathfrak{k}: \mathbb{R}^{+} \rightarrow \mathbb{R}^{+}.$
We define the \dfn{type-$\mathfrak{k}$ Julia-Fatou quotient} to be 
    \beq\label{jq}\fkf(z) = \frac{\IM f(z)}{\mathfrak{k}(\IM z)}.\eeq
Specific choices of the function $\mathfrak{k}$ recover the classical Fatou and Julia quotients on sets of nontangential approach, as well as intermediate behavior.

The averaged Julia-Fatou quotient can require more information about the boundary behavior of $f$ than can be recovered from a Stolz region. Hence, we define a larger family of sets that can be contained in regions that approach the boundary at $\tau$ tangentially, but with control over the approach in terms a boundary function $\la$.

Let $\lambda:[0,\infty)\rightarrow \R^{\geq 0}$ be a function. 
	We define a \dfn{$\lambda$-Stolz region at $\tau$} to be the set 
	\begin{align*} 
		S^\lambda_{\tau}= \big\{z\in D |  &\dist (z, \partial D) \geq \lambda(C), \\
		&\sqrt{\dist(z, \tau)^2 -\dist (z, \partial D)^2} \leq C, \text{for some } C>0\big\}.
	\end{align*}
	Note that a classical Stolz region with parameter $0 < M \leq 1$
	is a $\lambda$-Stolz region where $\lambda(t) = \left(\sqrt{\frac{M^2}{1-M^2}}\right)t.$

We will be concerned with understanding and classifying boundary regularity in terms of the Julia-Fatou quotient. Certain notions from the classical regime map directly into the Julia-Fatou setting. In analogy with the classical notion in Theorem \ref{jc}, we call $\tau \in \partial \Pi$ a \dfn{$B$-point} for $f$ if there exists a $\la$-Stolz region at $\tau$ such that  $\fkf$ is bounded as $z \to \tau$ through $S_\tau^\la$.

 Using Pick functions rather than their conformal relatives gives us access to Nevanlinna's eminently useful representation theorem \cite{nev22}, which states that any analytic $f:\Pi \to \cc\Pi$ must be of the form
\beq\label{fullnev}
		f(z) = a + bz + \int_\R \frac{1}{t-z} - \frac{t}{1+t^2} \dd\mu(t)
	\eeq
	for all $z \in \Pi$, where $a > 0, b \in \R$ and $\mu$ is a positive Borel measure on $\R$ with $\int_\R (1 + t^2)^{-1} \dd\mu < \infty$,
	and that moreover, any $f$ of such a form must be an analytic map from $\Pi$ to $\cc\Pi.$

The measure $\mu$ appearing in \eqref{fullnev} is a scalar multiple of the so-called \dfn{Nevanlinna measure} $\mu_f$ corresponding to $f$, defined as the weak limit
of the measures (in $x$) given by $\im f(x+iy)$ as $y \rightarrow 0.$ Where the meaning is unambiguous, we will often use drop the $f$ and write $\mu.$

The function $f$ may be used to study geometric features of the underlying measure $\mu$ and vice versa. Furthermore, via the spectral theorem, such functions can be used to study operator theory in general (a nice systematic overview can be found in Nikolski  \cite{nikolski1, nikolski2}) and perturbation theory in particular (see, e.g., Barry Simon's surveys of the foundations of the area in \cite{simonlec, simontrace}).

We adopt the approach of controlled averaging of Julia-Fatou type quotients initiated in \cite{ajq}.
Let $\lambda: \mathbb{R}^{+} \rightarrow \mathbb{R}^{\geq 0}.$ 
We define an \dfn{averaged Julia-Fatou type $\mathfrak{k}$ quotient with respect to $\lambda$ at the point $\tau$} to be:
$$A\mathfrak{F}^{\mathfrak{k},\tau}_{f,\lambda}(\eps) = \frac{1}{2\eps}\int^\eps_{-\eps}\fkf(\tau+x+i\lambda(\eps))dx.$$

Recall that we say that a function $\eta$ is $O(g(t))$ as $t \to 0$ if for some choice of $k >0$, there exists $\eps >0$ so that $\eta(t) \leq k g(t)$ for $\abs{t} < \eps$. We say that a function is $o(g(t))$ if for any choice of $k >0$ there exists $\eps >0$ so that $\eta(t) < k g(t)$ for $\abs{t} < \eps$. (In particular, $\eta$ is $o(1)$ at $0$ means that $\eta$ vanishes at $0$). Likewise, $\eta$ is $\Omega(g(t))$ as $t \to 0$ if for some choice of $k > 0$, there exists $\eps > 0$ so that $\eta(t)  > k g(t)$ for $\abs{t} < \eps$.

We obtain a so-called \emph{augur lemma} demonstrating the relationship between the measure density and boundedness of a Julia-Fatou type quotient whenever $\la(t)$ is $O(t)$ as $t\to 0$:

\begin{theorem}[Augur Lemma]\label{introaugur}
Let $f:\Pi \to \cc\Pi$ be analytic. Assume that $\la$ is $O(t)$. For sufficiently small $\eps > 0$,
\beq \label{augurinequality}
L_0 \frac{\mu(-\ep+\tau,\ep+\tau)}{(\mathfrak{k}\circ\lambda)(\eps)\ep} \leq A\mathfrak{F}^{\mathfrak{k},\tau}_{f,\lambda}(\eps)  \leq L_1 \frac{\mu(-2\ep+\tau,2\ep+\tau)}{(\mathfrak{k}\circ\lambda)(\eps)\ep} + L_2 \frac{\lambda(\eps)}{(\mathfrak{k}\circ\lambda)(\eps)\ep^2}
\eeq
where $L_i$ are constants which depend on $f,\mathfrak{k}, \lambda, \tau$ but are independent of $\eps.$
\end{theorem} 
Inequality \eqref{augurinequality} is especially nice when $\frac{\lambda(\eps)}{(\mathfrak{k}\circ \lambda)(\eps)\eps^2}$ is bounded as $\eps$ goes to $0$ as the $L_2$ term is bounded. 
The proof of the above theorem is the content of Section \ref{secaugur}.

 \begin{example}
  Let $\la(t) = t^\alpha$ where $\alpha > 1$ and $\k(t) = t^\beta$. The bad term in the augur inequality \eqref{augurineq} is
  \[
   \frac{\la(\eps)}{(\mk \circ \la)(\eps) \eps^2} = \eps^{\alpha(1 - \beta) - 2}
  \]
which will be bounded as $\eps \to 0$ when $\alpha(1 - \beta) - 2 > 0$. Notice that this provides a qualitative description of the relationship between the nature of the Julia-Fatou quotient (in terms of $\k$) and the width of the interval of integration (in terms of $\la$); in this case, for the quantity $\alpha\beta$ to stay large, a given $\beta$ may require a very large $\alpha$ for the augur inequality to tightly say that the measure $\mu$ has sub-$t^{\alpha\beta}$ density. 
 \end{example}

Let $\gandalf: \R^+ \to \R^+.$ We say that $f$ is \dfn{$\gandalf$-fortunate at $\tau$} if 
\begin{enumerate}
\item there exist $\mk$ and $\la$ such that $\gandalf(Ct) = (\mk \circ \la)(t)$ where $\la$ is $O(t)$, 
\item the quotient $\frac{\lambda(\eps)}{\gandalf(C\eps)\eps^2}$ is bounded as $\eps$ goes to $0$ for some $C>0$,
\item   and $\cc\lim_{\eps \to 0} A\mathfrak{F}^{\mathfrak{k},\tau}_{f,\lambda}(\eps) < \infty$.
\end{enumerate}
We say a measure $\mu$ has \dfn{sub-$\gandalf$-density at $\tau$} if $\frac{\mu(-\ep+\tau,\ep+\tau)}{\gandalf(C\eps)\eps}$
is bounded as $\eps \to 0$ for some $C>0.$
(We employ the $C$ in our definition so that our notions will be scale invariant.)

An immediate consequence of Theorem \ref{introaugur} gives a relationship between $f$ being $\gandalf$-fortunate
and the measure $\mu$ having sub-$\gandalf$-density. Moreover, subject to the constraints of the augur inequality, the decomposition $\gandalf(Ct) = (\mk \circ \la)(t)$ can be essentially arbitrary.
\begin{corollary} \label{fortunedensity}
Let $\gandalf: \R^+ \to \R^+.$
Let $f:\Pi \to \cc\Pi$ be analytic. The function  $f$ is $\gandalf$-fortunate at $\tau$ if and only if the Nevanlinna measure $\mu$ corresponding to $f$ has sub-$\gandalf$-density at $\tau.$

Moreover, in such a case, there exists $C>0$  such that for every $\mk, \la$ satisfying
\begin{enumerate}
    \item $\gandalf(Ct) = (\mk \circ \la)(t)$, 
    \item $\frac{\lambda(\eps)}{\gandalf(C\eps)\eps^2}$ is bounded as $\eps$ goes to $0$, 
    \item and $\la(t)$ is $O(t)$ as $t\rightarrow 0,$
\end{enumerate}    
we have that $\cc\lim_{\eps \to 0} A\mathfrak{F}^{\mathfrak{k},\tau}_{f,\lambda}(\eps) < \infty$. 

The converse direction relies on Lemma \ref{sarumandecomposition}, which states that any $\gandalf$ can be decomposed as $\gandalf = \k \circ \la$ with the desired properties.
\end{corollary}

 \begin{example}
	Corollary \ref{fortunedensity} says that
		\[\cc{\lim_{\eps \rightarrow 0}} \frac{\mu((-\eps,\eps))}{2\eps} < \infty\]
	if and only if 
		\[ \cc{\lim_{\eps \rightarrow 0}} \int^\eps_{-\eps} \IM f(t+i\eps^{2}) \dd t <\infty. \]
	 (Here take $\gandalf(t)=\k(t)=1$ and $\la(t)=t^{2}.$)
	Note the relation of the first limit to the Lesbesgue differentiation theorem and Radon-Nikodym theorem.

	More generally, let $\nu$ be a locally finite measure on $\mathbb{R}$ which is supported at $0.$
	Then, there exists a $C>0$ such that
		\[\cc{\lim_{\eps \rightarrow 0}} \frac{\mu((-\eps,\eps))}{\nu((-C\eps,C\eps))} < \infty\]
	if and only if
		there exists a $K>0$ such that
		\[ \cc{\lim_{\eps \rightarrow 0}} \int^\eps_{-\eps} \frac{\IM f(t+i\eps^{2}\nu(-K\eps,K\eps))}{\nu((-K\eps,K\eps))}\eps \dd t <\infty. \]
	That is, we can measure the relative density of two measures.
 \end{example}

We say a function $\gandalf: \R^+ \to \R^+$ is a \dfn{$\gamma$-augury} if there exists a $C>0$ such that
    $$\frac{t\gandalf(Ct)\dd\gamma(t)}{\gamma(t)^2}$$
is integrable on $[0,1).$ (Here $\dd\gamma(t)$ denotes the distributional derivative of $\gamma.$)

Let $\gamma: [0,\infty)\rightarrow \mathbb{R}^{\geq 0}$ be a monotone increasing function such that $\gamma(t)$ is $O(1)$ as $t \to 0$.
We say an analytic function $f: \Pi \rightarrow \cc \Pi$
is \dfn{$\gamma$-regular at $\tau$} whenever there exists a $C>0$ such that $\frac{1}{\gamma(C|t-\tau|)}$ is integrable on a neighborhood of $\tau$
with respect to the Nevanlinna measure $\mu_f.$ 

We obtain the following theorem relating $\gamma$-regularity of the measure $\mu$ to the $\gandalf$-fortune of $f.$
\begin{theorem}\label{introregfort}
Let $\gamma$ be $O(1)$ and $f:\Pi \to \cc\Pi$ be analytic. The function $f$ is $\gamma$-regular at $\tau$ if and only if there exists a $\gamma$-augury $\gandalf$ such that $f$ is $\gandalf$-fortunate at $\tau$.
\end{theorem}

We prove the above theorem as Theorem \ref{thmregfort}.

\subsection{The foliation}

\subsubsection{Application to spectral theory}
{The present endeavor is motivated by the connection between $\gamma$-regularity and the related tangential analysis of boundary behavior of Pick functions in \cite{ajq} and operator perturbation theory, particularly in the the work of Liaw and Treil \cite{lt2016, lt2017, lt2019}. For a self-adjoint operator $A$ on a separable Hilbert space $\hilbert$, one can define the function
\[
 F(z) = \ip{(A - zI)\inv \ph}{ \ph}= \int_\R \frac{d\mu(t)}{t - z},
\]
where the integral representation is the \emph{Cauchy transform} of $\mu$. The boundary behavior of $F$ contains information about the spectrum of $\mu$. In particular, the singular set of $\mu$ is supported on a set where the nontangential boundary limits of $\IM F$ are infinite.

A rank one perturbation $A_\alpha$ for a parameter $\alpha \in \R$ is the operator
\[
 A_\alpha = A + \alpha\ip{\cdot}{\ph}_\hilbert \ph.
\]
The theory of rank one perturbations depends on the observation that the functions $F_\alpha(z) = \int_\R \frac{\mu_\alpha(t)}{t - z}$ associated with $A_\alpha$ have a simple relation with the function $F$ associated with $\alpha$, the so-called \emph{Aronszajn-Krein formula}:
\[
 F_\alpha = \frac{F}{1 + \alpha F}.
\]
Rank one perturbation theory is in terms of the classical Julia-Fatou theory -- our motivation is to extend the more refined boundary behavior established in \cite{ajq} to understand the spectrum of the measures $\mu$ associated with Pick functions $F$.
}

\subsubsection{Spectral foliation}

{ Recall that by the Lebesgue Decomposition Theorem, a Borel measure $\mu$ on $\R$ can be decomposed as 
\[
 \mu = \mu_{\rm ac} + \mu_{\rm sc} + \mu_{\rm pp}
\]
corresponding to the absolutely continuous, singular continuous, and point parts of $\mu$.}

{ When $\mk$ is $O(1)$ and $\la$ is $\Omega(t)$ as $t\to 0$, and $\lim_{z \nt 0 } \mathfrak{F}_f^1(z)$ exists, then for any surjective M\"obius transformation $M$ of the upper halfplane to itself which is nonsingular at $f(\tau)$, then 
\beq\label{conformalajf}
 \cc\lim_{\eps \to 0} A\mathfrak{F}^{\mathfrak{k},\tau}_{f,\lambda}(\eps) < \infty \text{ iff }  \cc\lim_{\eps \to 0} A\mathfrak{F}^{\mathfrak{k},\tau}_{(M\circ f),\lambda}(\eps) < \infty.
\eeq
Since the finiteness of this quantity is (essentially) conformally invariant, we can use this quantity to foliate the spectrum of $\mu$, including the singular part. As was noted in \cite{lt2017}, the points such that $\lim_{z \nt 0 } f(z)$ does not exist will be immaterial to the spectrum. Accordingly, these points are called the \dfn{ethereal spectrum}, denoted $\ethereal_\mu$. Points which correspond to poles on the real line or $B$-points behave qualitatively similarly and comprise the \dfn{Julia spectrum}, which contains and behaves qualitatively like the point spectrum. Points such that $\lim_{z \nt 0 } \mathfrak{F}_f^1(z)$ is positive represent the \dfn{true ac spectrum}. }

{
Finally, points such that $\lim_{z \nt 0 } \mathfrak{F}_f^1(z)$ is zero or infinity which are not $B$-points or poles constitute the \dfn{cryptospectrum}, denoted $\enigma_\mu$, which notably contains and behaves qualitatively like the non-point singular spectrum. The union of the cryptospectum and ethereal spectrum can be further divided up into \dfn{$\mk,\la$-enigmas}, which are the sets
 
\beq \label{enigmadef}
\enigma_{\mu, \la}^\k := \left\{\tau : \cc\lim_{\eps \to 0} A\mathfrak{F}^{\mathfrak{k},\tau}_{f,\lambda}(\eps) < \infty \text{ OR } \cc\lim_{\eps \to 0} A\mathfrak{F}^{\mathfrak{k},\tau}_{-1/f,\lambda}(\eps) < \infty\right\}.
\eeq 

The following theorem says that various pieces of the spectrum are conformally invariant. Specifially, enigmas with $\lambda$ that correspond to suitably narrow (but often tangential) 
$\lambda$-Stolz regions for integration of the corresponding averaged Julia-Fatou type quotient are conformally invariant.
\begin{theorem}\label{specinv}
 Let $f:\Pi \to \cc\Pi$ be analytic and $\mu$ the associated measure. 
\begin{enumerate}
\item The ethereal spectrum, the Julia spectrum, and the ac spectrum are conformally invariant.
\item Suppose that $\gamma$ is $O(t)$ and monotone. Suppose that $\la$ is $O(t)$ and  $\Omega(\gamma)$. If $\k\circ\la$ is a $\gamma$-augury, then $\enigma_{\mu, \la}^{\k}$ is conformally invariant.
\end{enumerate}
\end{theorem}

Part (1) of Theorem \ref{specinv} is classical. Part (2) is proved below as Theorem \ref{enigmainv}.

\begin{example}
 Let $\gamma(t) = t^\eta$ where $\eta > 1$, $\la(t) = t^\alpha$ where $1 \leq \alpha < \eta$, and $\k = t^\beta$ with $\beta > \frac{\eta - 1}{\alpha}$. Then $$\enigma_{\mu_f, \la}^\k = \enigma_{\mu_{M \circ f}, \la}^\k.$$ 
\end{example}
Thus, one sees that, although the conditions for an enigma to be conformally invariant from Theorem \ref{specinv} are on face somewhat opaque, or one could even say enigmatic, for concrete choices of $\lambda$ it is often not difficult to derive concrete choices of $\k$ such that  $\enigma_{\mu_f, \la}^\k$ is conformally invariant. As we will see below, under these ``just-so" conditions we obtain ``horocyclic continuity" type results which allow us to conclude the function is regular enough to derive conformal invariance.

The following theorem says that for any point in the cryptospectrum, it is in some enigma.
\begin{theorem}
		Let $f:\Pi \to \cc\Pi$ be analytic and $\mu$ the associated measure.
    Let $\tau \in \enigma_\mu.$
    Let $\la$ be $\Omega(t).$
    There exists $\mk$ which is $o(1)$ such that $\tau \in \enigma^\mk_{\mu,\la}.$
\end{theorem}
The above theorem is proved as Theorem \ref{thmcryptoenigma}.

The following theorem, in light of Corollary \ref{fortunedensity} says that membership in an enigma implies some averaged Julia-Fatou type regularity of the function which is also reflected in the underlying measure.
\begin{theorem}
	Let $f:\Pi \to \cc\Pi$ be analytic and $\mu$ the associated measure.
     Let $\la$ be $\Omega(t).$
      Let $\mk$ be $o(1).$
    Let $\tau \in \enigma^\mk_{\mu,\la}.$
    Then $f$ is $\gandalf$-fortunate where
    $$\gandalf(t) = \frac{\la(t)(\mk \circ \la)(t)}{t}.$$
\end{theorem}

The above theorem is proved as Theorem \ref{thmenigmafortunate}.

\begin{example}
Suppose that $\tau \in \enigma_{\mu, \la}^\k$ where $\la(t) = t^\alpha$ for $0 < \alpha < 1$ and $\k(t) = t^\beta$ for $\beta > 0$. Then $f$ is $\gandalf$-fortunate at $\tau$ for $\gandalf(t) = t^{\alpha + \alpha\beta - 1}$. Let $\gamma(t) = t^\eta$. Then $\gandalf$ is a $\gamma$-augury when 
\[
\frac{t \gamma \dd \gamma}{\gamma^2} = \frac{t\, t^{\alpha + \alpha\beta -1} t^{\eta - 1}}{t^{2\eta}} = t^{\alpha + \alpha\beta - \eta - 1}
\]
is integrable on $[0,1)$, which will hold for
\[
\alpha + \alpha\beta - \eta - 1 > -1 \Rightarrow \eta < \alpha(\beta  +1).
\]
We also require that $1 \leq \eta < 2$. Thus, for fixed $\alpha, \beta$, we recover precisely the $t^\eta$-regularity of $f$.
\end{example}

\subsection{Horocyclic continuity and $\gamma$-regularity}

We establish the $\gamma$-analogue of \emph{horocyclic continuity} of a function. Classically, horocyclic continuity was established by Julia in \cite{ju20}. Let $\gamma_\alpha(t) = \gamma(\alpha t)$. We say a function is \dfn{$\gamma$-horocyclically continuous} whenever there is a fixed $\alpha >0$ so that
for each $\beta >0,$
		\[\sup_{S^{\beta \gamma_\alpha}_{\tau}\cap B(\tau, 1/\beta )} |f(z)-f(\tau)|<\infty,\]
and, moreover,
 $$\lim_{\beta \rightarrow \infty} \sup_{S^{\beta \gamma_\alpha}_{\tau}\cap B(\tau, 1/\beta )} |f(z)-f(\tau)| = 0.$$

We show that any $\gamma$-regular function is $\gamma$-horocyclically continuous.

\begin{theorem}
Let $f:\Pi \to \cc\Pi$ be analytic. Let $\gamma$ be monotone and $O(t)$. If $f$ is $\gamma$-regular then $f$ is $\gamma$-horocyclically continuous.
\end{theorem}

The above theorem is proved as Theorem \ref{thmregcont}.

\begin{example}
The case $\gamma(t)=t^2$ corresponds to the classical horocyclic continuity results in the Julia-Carath\'eodory theorem \cite{ju20, car29,amy10a}
\end{example}

\black

\section{The Augur Lemma}\label{secaugur}

\begin{lemma}\label{doubleintlemma}
Let $f:\Pi \to \cc\Pi$ be analytic and $\mu$ the associated measure.
\begin{align*}
 \afkf(\eps) &= \frac{1}{2\eps (\mk \circ \la)(\eps)} \int_\R \int_{-\eps}^\eps \frac{\la(\eps)}{(x - t)^2 + \la(\eps)^2} \dd x \dd \mu(t) \\
&= \frac{1}{2\eps (\mk \circ \la)(\eps)}\int_{[-\eps,\eps]} \tan^{-1} \left( \frac{\eps - t}{\la(\eps)} \right) - \tan^{-1} \left(\frac{-\eps - t}{\la(\eps)}\right) \dd \mu(t).
\end{align*}
\end{lemma}
\begin{proof}
Assume without loss of generality that $\tau = 0$.
 By definition,
 \begin{align*}
 \afkf(\eps) &= \frac{1}{2\eps} \int_{-\eps}^\eps \fkf(x + i\la(\eps)) \dd x \\
 &= \frac{1}{2\eps} \int_{-\eps}^{\eps} \frac{\IM f(x + i\la(\eps))}{\mk(\la(\eps))} \dd x \\
 &= \frac{1}{2\eps (\mk \circ \la)(\eps)} \int_\R \int_{-\eps}^\eps \frac{\la(\eps)}{(x - t)^2 + \la(\eps)^2} \dd x \dd \mu(t) \\
&=  \frac{1}{2\eps (\mk \circ \la)(\eps)}\int_{[-\eps,\eps]} \tan^{-1} \left( \frac{\eps - t}{\la(\eps)} \right) - \tan^{-1} \left(\frac{-\eps - t}{\la(\eps)}\right) \dd \mu(t) .
 \end{align*}
\end{proof}

\begin{lemma}\label{leftaugurlemma}
Let $f:\Pi \to \cc\Pi$ be analytic and $\mu$ the associated measure. For sufficiently small $\eps > 0$,
\begin{enumerate}
\item If $\la$ is $O(t)$, then , 
 \beq
  \afkf(\eps) \geq C \frac{\mu(-\eps, \eps)}{\eps(\mk \circ \la)(\eps) }.
 \eeq
Namely, if 
\[
\cc{\lim}_{\eps \to 0} \afkf(\eps) < \infty,
\]
then $f$ has sub-$\mk \circ \la$-density.
\item If $\la$ is $\Omega(t)$, then ,
\beq
 \afkf(\eps) \geq  C\frac{\mu(-\eps,\eps)}{\la(\eps) (\mk \circ \la)(\eps)},
\eeq

\end{enumerate}
\end{lemma}
\begin{proof} 

By Lemma \ref{doubleintlemma},
\[
 \afkf(\eps) = \frac{1}{2\eps (\mk \circ \la)(\eps)}\int_{[-\eps,\eps]} \tan^{-1} \left( \frac{\eps - t}{\la(\eps)} \right) - \tan^{-1} \left(\frac{-\eps - t}{\la(\eps)}\right) \dd \mu(t) 
\]

{\bf Part (1)}: Assume that $\la$ is $O(t)$.
Assume without loss of generality that $\tau = 0$. 

Then, 
\[
 \afkf(\eps) \geq  L_0 \frac{\mu(-\eps, \eps)}{(\mk \circ \la)(\eps) \eps},
\]
since 
\[
\tan^{-1} \left( \frac{\eps - t}{\la(\eps)} \right) - \tan^{-1} \left(\frac{-\eps - t}{\la(\eps)}\right)
\]
is bounded below when $\eps$ is near $0$.

{\bf Part (2)}: Assume that $\la$ is $\Omega(t)$.  
As
\[
C\left(\frac{\eps - t}{\la(\eps)} - \frac{-\eps - t}{\la(\eps)}\right) \leq \tan^{-1} \left( \frac{\eps - t}{\la(\eps)} \right) - \tan^{-1} \left(\frac{-\eps - t}{\la(\eps)}\right)
\]
for some $C>0$ when $\eps$ is near $0$ since the derivative of arctangent is nonvanishing and the quantities $ \frac{\eps - t}{\la(\eps)}$ and $\frac{-\eps - t}{\la(\eps)}$ are in some compact interval, we have
\begin{align*}
 \afkf(\eps) &\geq  C\frac{1}{\eps (\mk \circ \la)(\eps)}\int_{[-\eps,\eps]} \frac{\eps - t}{\la(\eps)} - \frac{-\eps - t}{\la(\eps)} \dd \mu(t) \\
 &\geq 
 C\frac{\mu(-\eps,\eps)}{\la(\eps) (\mk \circ \la)(\eps)},
\end{align*}

\end{proof}

\begin{lemma}\label{rightaugurlemma}
Let $f: \Pi \to \cc{\Pi}$ be analytic and $\mu$ the associated measure. Let $\mathfrak{k}:\R^+ \to \R^+$. Assume that $\la$ is $O(t)$. Then 
\[
\afkf(\eps) \leq L_1 \frac{\mu(-2\eps, 2\eps)}{(\mk \circ \la)(\eps)\eps} + L_2 \frac{\la(\eps)}{(\mk \circ \la)(\eps) \eps^2}.
\]
\end{lemma}
\begin{proof}
By Lemma \ref{doubleintlemma},
\begin{align*}
 &\afkf(\eps) \\ 
 &\leq \frac{1}{2\eps(\mk \circ \la)(\eps)} \int_{[-2\eps, 2\eps]} \tan^{-1} \left( \frac{\eps - t}{\la(\eps)} \right) - \tan^{-1} \left(\frac{-\eps - t}{\la(\eps)}\right) \dd \mu(t) \\
 &\hspace{.5in} + \frac{\la(\eps)}{2\eps (\mk \circ \la)(\eps)} \int_{[-2\eps, 2\eps]^c} \int_{-\eps}^\eps \frac{1}{(x - t)^2 + \la(\eps)^2} \dd x \dd \mu(t) \\
&\leq L_1 \frac{\mu(-2\eps, 2\eps)}{(\mk \circ \la)(\eps)\eps} +  \frac{\la(\eps)}{2\eps (\mk \circ \la)(\eps)} \int_{[-2\eps, 2\eps]^c} \int_{-\eps}^\eps \frac{1}{(x - t)^2 + \la(\eps)^2} \dd x \dd \mu(t) \\
&\leq L_1 \frac{\mu(-2\eps, 2\eps)}{(\mk \circ \la)(\eps)\eps} +  C\frac{\la(\eps)}{ (\mk \circ \la)(\eps)} \int_{[-2\eps, 2\eps]^c} \frac{1}{t^2} \,  \dd \mu(t) \\
&\leq L_1 \frac{\mu(-2\eps, 2\eps)}{(\mk \circ \la)(\eps)\eps} +  C\frac{\la(\eps)}{ (\mk \circ \la)(\eps)} \int_{[-2\eps, 2\eps]^c} \frac{1 + t^2}{t^2} \,  (1 + t^2)\inv\dd \mu(t) \\
&\leq L_1 \frac{\mu(-2\eps, 2\eps)}{(\mk \circ \la)(\eps)\eps} +  \tilde{C}\frac{\la(\eps)}{ (\mk \circ \la)(\eps)} \int_{[-2\eps, 2\eps]^c} \frac{1}{\eps^2}\,  (1 + t^2)\inv\dd \mu(t) \\
 & \leq L_1 \frac{\mu(-2\eps, 2\eps)}{(\mk \circ \la)(\eps)\eps} + L_2 \frac{\la(\eps)}{(\mk \circ \la)(\eps) \eps^2},
\end{align*}
which establishes the claim.
\end{proof}

\begin{theorem}[The Augur Lemma]\label{augurlemma}
{ Let $f:\Pi \to \cc\Pi$ be analytic and $\mu$ the associated measure. Let $\mathfrak{k}:\R^+ \to \R^+$. Assume that $\la$ is $O(t)$. For  $\eps > 0$ sufficiently small,} 
 \beq\label{augurineq}
  L_0 \frac{\mu(-\eps, \eps)}{(\mk \circ \la)(\eps) \eps} \leq \afkf(\eps) \leq L_1 \frac{\mu(-2\eps, 2\eps)}{(\mk \circ \la)(\eps)\eps} + L_2 \frac{\la(\eps)}{(\mk \circ \la)(\eps) \eps^2}
 \eeq
\end{theorem}
\begin{proof}
The result is part (1) of Lemma \ref{leftaugurlemma} combined with Lemma \ref{rightaugurlemma}.

 \end{proof}

\section{Regularity consequences}

We first prove a set-theoretic lemma that requires the Axiom of Choice for $\R$.

\begin{lemma}\label{settheory}
Suppose that $\eta: \R^+ \to \R^+$. Then there exists an injective $\tilde\eta:\R^+ \to \R^+$ so that
\[
1 \leq \frac{\tilde\eta}{\eta} \leq 2.
\]
\end{lemma}

\begin{proof}
Let $V$ be a collection of coset representatives for $\R/\mathbb{Q}$. (That is, $V$ is the so-called Vitali set.) Choose a bijection $\xi: \R^+ \to V$. For each $x \in \R^+$, choose $\tilde\eta(x)$ in 
$\mathbb{Q} + \xi(x)$ such that $\eta(x) \leq \tilde\eta(x) \leq 2 \eta(x)$. Note that we can find such an $\tilde\eta(x)$ as the interval $[\eta(x), 2\eta(x)]$ has non-empty interior and $\mathbb{Q} + \xi(x)$ is dense. The function $\tilde\eta$ is injective as the sets $\mathbb{Q} + \xi(x)$ are disjoint as they are the cosets of $\mathbb{Q}$ in $\R$ which therefore partition $\R$.
\end{proof}

\begin{lemma}\label{sarumandecomposition}
Let $\gandalf: \R^+ \rightarrow \R^+.$
There are $\lambda$ and $\mk$ such that
$ \gandalf = \mk \circ \la$ and 
$\frac{\lambda(t)}{\gandalf(t)t^2}$ is bounded as $t \rightarrow 0,$
and $\la(t)$ is $O(t)$ as $t \to 0.$
\end{lemma}
\begin{proof}
    Define $\kappa(t) = \min\{\gandalf(t)t^2,t\}.$
    Note $\kappa$ is $O(t).$
    Note $\frac{\kappa(t)}{\gandalf(t)t^2} \leq 1.$ 
    By Lemma \ref{settheory}, we can choose an injective positive function $\lambda$
    such that $1\leq \lambda / \kappa\leq 2.$
    
    Now define $\mk$ such that $ \gandalf = \mk \circ \la$ and we are done.
\end{proof}

\begin{theorem}\label{thmregfort}
Let $\gamma$ be $O(1)$ and $f:\Pi \to \cc\Pi$ be analytic. The function $f$ is $\gamma$-regular at $\tau$ if and only if there exists a $\gamma$-augury $\gandalf$ such that $f$ is $\gandalf$-fortunate at $\tau$.
\end{theorem}
\begin{proof}
Without loss of generality, $\tau =0.$
For both parts of the equivalence, we will use the fact that 
$$\int_{[-1,1]} \frac{1}{\gamma(|t|)} \dd \mu(t) = 
\int_{[-1,1]} \frac{\mu([-t,t])}{\gamma(|t|)^2} \dd\gamma(t) + \mu([-1,1])\gamma(1)$$
by the layer cake principle.

$(\Rightarrow)$
Suppose $f$ is $\gamma$-regular at $\tau.$
Without loss of generality assume
$$\int \frac{1}{\gamma(|t|)} \dd \mu(t)<\infty.$$
So, we know that 
 $$\int_{[-1,1]} \frac{\mu([-t,t])}{\gamma(|t|)^2} \dd\gamma(t) < \infty.$$
Let $\gandalf(t) = \frac{\mu(-t,t)}{t}.$
Now $$\frac{\gandalf(t)\dd \gamma(t)}{t\gamma(t)^2} = \frac{\mu([-t,t])}{\gamma(|t|)^2} \dd\gamma(t)$$
is integrable so $\gandalf$ is a $\gamma$-augury.
Choosing $\lambda, \mk$ as in Lemma \ref{sarumandecomposition}, we see by the Augur lemma that $\mu$ has sub-$\gandalf$-density. Therefore by Corollary \ref{fortunedensity}, we see that $f$ is $\gandalf$-fortunate.

$(\Leftarrow)$ On the other hand suppose that there exists a $\gamma$-augury $\gandalf$ such that $f$ is $\gandalf$-fortunate at $\tau.$
So, by Lemma \ref{fortunedensity}, we see that $\mu$ has sub-$\gandalf$-density.
Without loss of generality,
$\frac{\mu(-t,t)}{\gandalf(t)t}$ is bounded.
So now,
$$\int_{[-1,1]} \frac{\mu([-t,t])}{\gamma(|t|)^2} \dd\gamma(t) \leq 
\int_{[-1,1]} \frac{\gandalf(t)}{t\gamma(|t|)^2} \dd\gamma(t)$$
which is integrable by our assumption that $\gandalf$ is a $\gamma$-augury
so we are done.
\end{proof}

\section{The foliation}\label{foliation}

\begin{theorem}\label{thmcryptoenigma}
		Let $f:\Pi \to \cc\Pi$ be analytic and $\mu$ the associated measure.
    Let $\tau \in \enigma_\mu.$
    Let $\la$ be $\Omega(t).$
    There exists $\mk$ which is $o(1)$ such that $\tau \in \enigma^\mk_{\mu,\la}.$
\end{theorem}
\begin{proof}
    Without loss of generality, $\tau = 0,$ and the nontangential limit as $f$ goes to $\tau$ is $0.$
    Define
        $$\mk(t) =\frac{1}{2t}\int^t_{-t} \IM f(x+it)dx.$$
    Since the nontangential limit of $f$ as we approach $\tau$ exists and is $0$, this implies 
    that $\mk$ is $o(1).$
    A quick calculation gives that $\afkf(\eps)$ is bounded as $\eps$ goes to $0.$
\end{proof}

\begin{theorem}\label{thmenigmafortunate}
	Let $f:\Pi \to \cc\Pi$ be analytic and $\mu$ the associated measure.
     Let $\la$ be $\Omega(t).$
      Let $\mk$ be $o(1).$
    Let $\tau \in \enigma^\mk_{\mu,\la}.$
    Then $f$ is $\gandalf$-fortunate where
    $$\gandalf(t) = \frac{\la(t)(\mk \circ \la)(t)}{t}.$$
\end{theorem}
\begin{proof}

By assumption, $\tau \in \enigma^\mk_{\mu,\la}$, and so by definition $ A\mathfrak{F}^{\mathfrak{k},\tau}_{f,\lambda}(\eps)$ is bounded.

Since $\la$ is $\Omega(t)$, by Lemma \ref{leftaugurlemma} part(2), we have
\[
 \afkf(\eps) \geq   C\frac{\mu(-\eps,\eps)}{\la(\eps) (\mk \circ \la)(\eps)}.
\]

Note that $\mu$ has sub-$\gandalf$-density as 
\[
\frac{\mu(-\eps, \eps)}{\gandalf(\eps) \eps} = \frac{\mu(-\eps, \eps)}{\la(\eps) (\mk \circ \la)(\eps)} \leq  \afkf(\eps) < \infty.
\]

Then by Corollary \ref{fortunedensity}, $f$ is $\gandalf$-fortunate.
\end{proof}

\black
\section{Horocyles}

%

\begin{lemma}\label{boundinglemma}
Let $\gamma$ be $O(t)$.
\[
\sup_{S^{\beta\gamma} \cap B(\frac{1}{\beta})} \abs{\frac{\gamma(t)}{t - z} - \frac{\gamma(t)}{t}} < \max \{2C, \frac{1 + \beta C}{\beta}\},
\]
where $C$ is the implicit order constant.
\end{lemma}

\begin{proof}
 For each $t$, for $z = x + i\beta\gamma(2x)$ (that is, on the boundary of $S^{\beta\gamma}$),
\begin{align*}
\abs{\frac{\gamma(t)}{t - z}  -\frac{\gamma(t)}{t}} &= \abs{\frac{z\gamma(t)}{t(t - z)}} \\
&= \abs{\frac{\gamma(t)}{t}} \abs{\frac{z}{t-z}}
\end{align*}
Consider the case where $|t-x|\geq x/2.$
Then, $\abs{\frac{z}{t-z}} \leq 2.$
On the other hand, if $|t-x|\leq x/2,$
then,
$\abs{\frac{z}{t-z}} \leq \frac{(1+\beta C)x}{\beta\gamma(2x)}
\leq \frac{(1+\beta C)3t}{\beta\gamma(t)},$
as $t$ is at least $x/2.$
\end{proof}

{
\begin{theorem}\label{thmregcont}
Let $f:\Pi \to \cc\Pi$ be analytic. Let $\gamma$ be monotone and $O(t)$. If $f$ is $\gamma$-regular then $f$ is $\gamma$-horocyclically continuous.
\end{theorem} 
\begin{proof}
  Without loss of generality, assume that $\tau = 0$ and that $\gamma$ is upper semi-continuous.  It suffices to consider the functions $\frac{\gamma(t)}{t - z} - \frac{\gamma(t)}{t}$ as they are the extreme points of the set of Pick functions taking $0 \to 0$ that are $\gamma$-regular.
Write $f$ as 
\[
 f(z) = c + \int \frac{\gamma(t)}{t - z} - \frac{\gamma(t)}{t} \, \dd\mu(t)
\]
where $\mu$ is supported on $[-1,1]$ and $\frac{1}{\gamma(t)}$ is $\mu$-integrable. Then
\beq
 \sup_{S^{\beta\gamma}\cap B(\frac{1}{\beta})} \abs{f(z) - c} \leq \int \sup_{S^{\beta\gamma}\cap B(\frac{1}{\beta})} \abs{\frac{\gamma(t)}{t - z} - \frac{\gamma(t)}{t}} \,\dd\mu.
\eeq
By Lemma \ref{boundinglemma}, the integrand is bounded and pointwise convergent to 0 in $\beta$ for each $t$. Then the dominated convergence theorem gives the conclusion.

\end{proof}
}

\begin{theorem}\label{enigmainv}
 Let $f:\Pi \to \cc\Pi$ be analytic and $\mu$ the associated measure. Suppose that $\gamma$ is $O(t)$ and monotone. Suppose that $\la$ is $O(t)$ and  $\Omega(\gamma)$. If $\k\circ\la$ is a $\gamma$-augury, then $\enigma_{\mu, \la}^{\k}$ is conformally invariant.
\end{theorem}
\begin{proof}
  Let $M$ be a M\"obius transform nonsingular at $\tau$. Let $\tau \in \enigma_{\mu,\la}^\k.$ Without loss of generality, assume that $$\cc\lim_{\eps \to 0} A\mathfrak{F}^{\mathfrak{k},\tau}_{f,\lambda}(\eps) < \infty.$$
  As $\k \circ \la$ is a $\gamma$-augury, Theorem \ref{thmregfort} implies that $f$ is $\gamma$-regular at $\tau$ since Lemma \ref{leftaugurlemma} part (1) implies that $f$ has sub-$\gandalf$-density which by Corollary \ref{fortunedensity} is equivalent to $f$ being $\gandalf$-fortunate. Without loss of generality, $f$ is bounded on $S^\gamma \cap B(\frac{1}{d})$, and therefore conformal maps look essentially linear and thus the corresponding quotients for $M\circ f$ and $f$ are essentially linearly related.
\end{proof}

}

\bibliography{references}
\bibliographystyle{plain}

\end{document}